\documentclass[a4paper, reqno]{amsart} 
\usepackage{graphicx} 

\usepackage{amsmath,amsthm, amssymb, amsfonts}
\usepackage[margin=3cm]{geometry} 
\usepackage{url}

\newtheorem{theorem}{Theorem}[section]
\newtheorem{lemma}[theorem]{Lemma}

\newtheorem{proposition}[theorem]{Proposition}

\theoremstyle{definition}
\newtheorem{definition}[theorem]{Definition}

\newtheorem{problem}[theorem]{Problem}

\theoremstyle{remark}

\usepackage{float}
\restylefloat{table}
\restylefloat{figures}
\usepackage{multirow}
\usepackage{graphicx}
\usepackage{tabularx}
\usepackage{color}
\usepackage{here}
\usepackage{algpseudocode,algorithm}

\usepackage{fourier} 

\definecolor{mygray}{gray}{0.6}
\definecolor{lg}{gray}{0.88}
\definecolor{pakistan}{rgb}{0.0, 0.5, 0.0}
\definecolor{kimidori}{rgb}{0.85,0.93,0.3}
\definecolor{mypink}{rgb}{0.9, 0.0, 0.4}
\definecolor{yamabuki}{rgb}{1.0, 0.86, 0.0}
\definecolor{navy}{rgb}{0.0,0.0,0.7}
\definecolor{darkred}{rgb}{0.7,0.0,0.0}

\usepackage{tcolorbox}
\tcbuselibrary{breakable} 

\title[Ranking top-$k$ trees in tree-based phylogenetic networks] 
{Ranking top-$k$ trees in tree-based phylogenetic networks}
\author{Momoko Hayamizu$^{1, 2}$}
\thanks{$^1$ The Institute of Statistical Mathematics}
\thanks{$^2$ Japan Science and Technology Agency (JST) PRESTO}
\address{The Institute of Statistical Mathematics, Tokyo, Japan}
\email{hayamizu@ism.ac.jp}
\author{Kazuhisa Makino$^{3}$}
\thanks{$^3$ Research Institute for Mathematical Sciences, Kyoto University}
\address{Research Institute for Mathematical Sciences, Kyoto University, Kyoto, Japan}
\email{makino@kurims.kyoto-u.ac.jp}



\begin{document}
\begin{abstract}
`Tree-based' phylogenetic networks proposed by Francis and Steel have attracted much attention of  theoretical biologists in the last few years. At the heart of the definitions of tree-based phylogenetic networks is the notion of `support trees', about which there are numerous algorithmic problems that are important for evolutionary data analysis.  Recently, Hayamizu (arXiv:1811.05849 [math.CO]) proved a structure theorem for tree-based phylogenetic networks and obtained linear-time and linear-delay algorithms for many basic problems on support trees, such as counting, optimisation, and enumeration.  In the present paper, we consider the following fundamental problem in statistical data analysis: given a tree-based phylogenetic network $N$ whose arcs are associated with probability, create the top-$k$ support tree ranking for $N$ by their likelihood values.  
We provide a linear-delay (and hence optimal) algorithm for the problem and thus reveal the interesting property of tree-based phylogenetic networks that ranking top-$k$ support trees is as computationally easy as picking $k$ arbitrary support trees.

\end{abstract}

\maketitle


\section{Introduction}\label{sec:intro}
Although phylogenetic trees have been used as the standard model of evolution, phylogenetic networks have become popular amongst biologists as a tool to describe conflicting signals in data or uncertainty in evolutionary histories~\cite{FHV2018,  FS, Huson}. Therefore, when we wish to reconstruct the phylogenetic tree $T$ on a set $X$ of species from non-tree-like data, a natural idea would be to describe the data using a phylogenetic network $N=(V, A)$ on $X$ and then remove extra arcs to discover an embedding $\tau=(V, S)$ of $T$ inside $N$, where $\tau$ is called a `support tree' of $N$~\cite{FS}. 

However, the above strategy only makes sense when $N$ is `tree-based', namely, $N$ is merely a tree with additional arc~\cite{FS}, which is not always the case~\cite{LeoBlog}. In~\cite{FS},  Francis and Steel provided a linear-time algorithm for finding a support tree  of $N$ if $N$ is tree-based and reporting that it does not exist otherwise.  Another linear-time algorithm for this decision problem was obtained by Zhang in~\cite{LX}. 

While Francis and Steel's work was followed by many studies (\textit{e.g.},~\cite{Owen,fischer2018non, FHV2018, francis2018new, UTBN,pons,LX}), Hayamizu's recent work~\cite{MH2018structural} significantly advanced our understanding of how tree-based networks could be useful in contemporary phylogenetic analysis. In fact, Hayamizu's structure theorem has derived a series of linear-time and linear-delay algorithms for many basic  problems (\textit{e.g.}, counting, enumeration and optimisation) on support trees,  and has thus enabled various data analysis using tree-based phylogenetic networks (see~\cite{MH2018structural} for details).  

In the present paper, we consider a so-called `top-$k$ ranking problem', with the aim to further facilitate the application of  tree-based phylogenetic networks. The problem is as follows: given a tree-based phylogenetic network $N$ where each arc $a$  exists 
in the true evolutionary lineage 
with probability $w(a)>0$, list top-$k$ support trees of $N$ in non-increasing order by their likelihood values. 
We note that this problem is an important generalisation of the top-$1$ ranking problem, which asks for a maximum likelihood support tree of $N$ and can be solved in linear time~\cite{MH2018structural}, since nearly optimal support trees  can provide more biological insights than the maximum likelihood one.

At first glance, ranking top-$k$ support trees may seem more difficult than picking $k$ arbitrary  support trees, the latter of which is possible with linear delay~\cite{MH2018structural}; however, in this paper, we provide a linear-delay (\textit{i.e.}, optimal) algorithm for the top-$k$ ranking problem and thus reveal that the above two problems have the same time complexity, which is an interesting property of tree-based phylogenetic networks.

\section{Preliminaries}\label{sec:preliminaries}
Throughout this paper, $X$ represents a non-empty finite set of present-day species. All graphs considered here are finite, simple, directed acyclic graphs. 
For a graph $G$, $V(G)$ and $A(G)$ denote the sets of vertices and arcs of $G$, respectively. 
A graph $G$ is called a \emph{subgraph} of a graph $H$ if both   $V(G) \subseteq V(H)$ and $A(G) \subseteq A(H)$ hold, in which case we write $G \subseteq H$.  
When $G \subseteq H$ but  $G\neq H$, then $G$ is called a \emph{proper} subgraph of $H$. When $G \subseteq H$ and $V(G)=V(H)$, $G$ is a \emph{spanning} subgraph of $H$.
Given a graph $G$ and a non-empty subset $A^\prime$ of $A(G)$, $A^\prime$ is said to \emph{induce the subgraph  $G[A^\prime]$ of $G$}, that is, the one whose arc-set is $A^\prime$ and whose vertex-set consists of all ends of arcs in $A^\prime$. For a graph $G$  with $|A(G)| \geq 1$ and a partition  $\{A_1,\dots,A_{d}\}$ of $A(G)$, the collection  $\{G[A_1], \dots, G[A_{d}]\}$ of arc-induced subgraphs of $G$ is called a  \emph{decomposition}  of $G$. 
For an arc $a=(u,v)\in A(G)$, $u$ and $v$ are called the tail and head of $a$ and are denoted by $\it tail(a)$ and $\it head(a)$, respectively.
For a vertex $v$ of a graph $G$,  the \emph{in-degree of $v$ in $G$}, denoted by ${\it deg}^-_G(v)$, is defined to be the cardinality of the set $\{a\in A(G)\mid \it head(a)=v\}$. The \emph{out-degree of $v$ in $G$}, denoted by ${\it deg}^+_G(v)$,  is defined in a similar manner.  For any graph $G$, a vertex $v\in V(G)$ with  $({\it deg}^-_G(v), {\it deg}^+_G(v))=(1,0)$ is called a \emph{leaf} of $G$.

\begin{definition}\label{dfn:rbpn}
A \emph{rooted binary phylogenetic $X$-network} is defined to be a finite simple directed acyclic graph $N$ with the following properties:
\begin{enumerate}
  \item $N$ has a unique vertex $\rho$  with ${\it deg}^-_N(\rho)=0$ and ${\it deg}^+_N(\rho)\in \{1,2\}$; 
  \item $X$ is the set of leaves of $N$;
  \item for any $v\in V(N)\setminus (X\cup \{\rho\})$, $\{{\it deg}^-_N(v), {\it deg}^+_N(v)\}= \{1,2\}$ holds.
\end{enumerate}
\end{definition}
In Definition~\ref{dfn:rbpn},  the vertex $\rho$ is called \emph{the root} of $N$, and a vertex  $v\in V(N)$ with $({\it deg}^-_N(v), {\it deg}^+_N(v))= (2,1)$ is called a \emph{reticulation vertex} of $N$. 
When $N$ has no reticulation vertex, $N$ is called a \emph{rooted binary phylogenetic $X$-tree}.

\begin{definition}[\cite{FS}]
If a rooted binary phylogenetic $X$-network $N$ that has a spanning tree $\tau$ that can be obtained by inserting zero or more vertices into each arc of a rooted binary phylogenetic $X$-tree $T$, then $N$ is said to be \emph{tree-based} and $\tau$ is called a \emph{support tree} of $N$. 
\end{definition}

\begin{theorem}[\cite{FS}]\label{thm:bijection}
   Let $N$ be a rooted binary phylogenetic $X$-network and let $S$ be a subset  of $A(N)$. Then, the subgraph $N[S]$ of $N$ is a support tree of $N$ if and only if $S$ satisfies the following three conditions, in which case $S$ is called an `admissible' arc-set of $N$.  
Moreover, there exists a one-to-one correspondence between support trees of $N$ and admissible arc-sets of $N$. 
\begin{enumerate}
\item  $S$ contains all $(u,v)\in A(N)$ with ${\it deg}^-_{N}(v)=1$ or ${\it deg}^+_{N}(u)=1$.
\item for any  $a_1, a_2\in A(N)$ with ${\it head}(a_1)={\it head}(a_2)$, exactly one of $\{a_1, a_2\}$ is in $S$. 
\item for any $a_1, a_2\in A(N)$ with ${\it tail}(a_1)={\it tail}(a_2)$, at least one of $\{a_1, a_2\}$ is in $S$. 
\end{enumerate}  
\end{theorem}
In this paper, as the conditions in Theorem~\ref{thm:bijection} still make sense for any subgraph of $N$, we consider admissible arc-sets of subgraphs of $N$. 

\section{Known results: the structure of support trees}\label{sec:preliminaries1}
Here, we summarise without proofs the relevant material in~\cite{MH2018structural}. 
A connected subgraph $Z$ of  a tree-based phylogenetic $X$-network $N$ with $|A(Z)|\geq 1$  is called a  \emph{zig-zag trail} (\emph{in $N$}) if  there exists a permutation $(a_1,\dots,a_m)$ of $A(Z)$ such that for each $i\in [1,m-1]$, either ${\it head}(a_i)={\it head}(a_{i+1})$ or ${\it tail}(a_i)={\it tail}(a_{i+1})$ holds. 
Then, any zig-zag trail $Z$ in $N$ is specified by an alternating sequence of (not necessarily distinct) vertices and distinct arcs of $N$, such as $(v_0, (v_0, v_1), v_1, (v_2, v_1), v_2, (v_2, v_3), \dots,  (v_m, v_{m-1}), v_m)$, which can be more concisely expressed as  $v_0>v_1<v_2>v_3< \cdots > v_{m-1} < v_m$ or in reverse order. 
	A zig-zag trail $Z$ in $N$ is said to be  \emph{maximal} if $N$ contains no zig-zag trail $Z^\prime$  such that $Z$ is a proper subgraph of $Z^\prime$. 
A maximal zig-zag trail $Z$ with even $m:=|A(Z)|\geq 4$ is called a \emph{crown} if $Z$  can be written in the cyclic form $v_0 < v_1 > v_2 < v_3 > \cdots  > v_{m-2} < v_{m-1} > v_{m} = v_0$ and is called a \emph{fence} otherwise. Furthermore, a fence $Z$ with odd $|A(Z)|$ is called an \emph{N-fence}, in which case $Z$ can be expressed as $v_0 > v_1 < v_2 > v_3 < \cdots >  v_{m-2} < v_{m-1} > v_m$. A fence $Z$ with even $|A(Z)|$ is called an \emph{M-fence} if it can be  written in the form  $v_0 < v_1 > v_2 < v_3 > \cdots > v_{m-2} < v_{m-1} > v_m$, rather than $v_0 > v_1 < v_2 > v_3 < \cdots < v_{m-2} > v_{m-1} < v_m$. 
	
From now on, we represent a  maximal zig-zag trail $Z$  by a sequence $\langle a_1,\dots,a_{|A(Z)|}\rangle$ of the elements of $A(Z)$ that form the zig-zag trail in this order, assuming that no confusion arises. Then, we can encode an arbitrary arc-induced subgraph of $Z$ by an $|A(Z)|$-dimensional vector. For example, for an N-fence $Z=\langle a_1, a_2, a_3, a_4, a_5\rangle$, the subgraph of $Z$ induced by the subset $\{a_1, a_3, a_5\}\subseteq A(Z)$ is specified by the vector $(1\; 0\; 1\; 0\; 1)=(1 (01)^2)$. 
With this notation,  we can state  Hayamizu's structure theorem for tree-based phylogenetic networks, which gives an explicit characterisation of the family $\Omega$ of all admissible arc-sets of $N$ as follows.

\begin{theorem}[\cite{MH2018structural}]\label{structure}
Any tree-based phylogenetic $X$-network $N$ is uniquely decomposed into maximal zig-zag trails $Z_1,\dots, Z_d$, each of which is a crown, M-fence or N-fence. Moreover, a subgraph $G$ of $N$ is a support tree of $N$ if and only if $A(G)\cap A(Z_i)$ is an admissible arc-set of $Z_i$ for any $i\in [1,d]$. Furthermore, the collection $\Omega$ of support trees of $N$  is characterised by a direct product of families $\Omega_1,\dots,\Omega_d$ of the admissible arc-sets of $Z_1,\dots,Z_d$, namely, we have $\Omega=\prod_{i=1}^{d}{\Omega_i}$ with 
\begin{align*}\label{sequence}
\Omega_i := 
  \begin{cases}
      \bigl\{ ((01)^{|A(Z_i)|/2}),\ ((10)^{|A(Z_i)|/2})\bigr\} & \text{if $Z_i$ is a crown;} \\
      \bigl\{ (1(01)^{(|A(Z_i)|-1)/2}) \bigr\}  & \text{if $Z_i$ is an N-fence;} \\ 
    \bigl\{ (1(01)^p(10)^q 1) \mid  p,q\in \mathbb{Z}_{\geq 0}, p+q=(|A(Z_i)|-2)/2\bigr\} & \text{if $Z_i$ is an M-fence.
    } 
   \end{cases}
\end{align*}
\end{theorem}

\section{Top-$k$ support tree ranking problem}\label{sec:problem.description}
Given a tree-based phylogenetic $X$-network $N$ where each arc $a$ is chosen with probability $w(a)\in (0, 1]$, we  can assign a ranking number to each support tree $\tau\in \Omega$ of $N$ by the likelihood value $f(\tau):=\prod_{a\in A(\tau)}{w(a)}$. 
In principle, the top-$k$ support tree ranking problem for $N$ asks for an ordered set $\langle \tau^{(1)},\dots, \tau^{(k)}\rangle$ of $k$ support trees of $N$ such that  $f(\tau^{(1)})\geq \dots \geq f(\tau^{(k)})\geq f(\tau)$ holds for  any support tree $\tau$ of $N$  other than $\tau^{(i)}$ ($i=1,\dots, k$). However, such a ranking is not unique in general, since there can be `ties' in the collection $\Omega$ of support trees of $N$ as well as in the family $\Omega_i$ of admissible arc-sets of each maximal zig-zag trail $Z_i$ in $N$. 
For convenience, we ensure the uniqueness of the ranking by using the lexicographical order $\leq_{\mathrm{lex}}$ on vectors as follows.

Assume that  $N$ is a tree-based phylogenetic $X$-network with $\Omega=\prod_{i=1}^{d}{\Omega_i}$ as in Theorem~\ref{structure} and that $Z_i$ is any maximal zig-zag trail in $N$. We define the \emph{local ranking for $Z_i$}  to be a totally ordered set $(\Omega_i, \leq^*)$ such that for any $x, y \in \Omega_i$, $x \leq^* y$ holds if  either    $f(x)> f(y)$ or $(f(x) = f(y) \land x \leq_{\mathrm{lex}} y)$ holds. 
Note that the elements of $\Omega_i$ are $|A(Z_i)|$-dimensional vectors and any two of them are comparable lexicographically. From now, we identify the $j$-th element of $(\Omega_i, \leq^*)$ with its local ranking number $j\in \{1,\dots, |\Omega_i|\}$ in order to write $\Omega=\prod_{i=1}^{d}{\{1,\dots, |\Omega_i|\}}$. Then, the elements of $\Omega$ are vectors having the same dimension  again and so we can break ties by using $\leq_{\mathrm{lex}}$  as before. Abusing the notation $\leq^*$ slightly, we call  the totally ordered set $(\Omega, \leq^*)$ the \emph{support tree ranking} (\emph{for $N$}). For any $k\in \mathbb{N}$ with $k\leq |\Omega|$, the \emph{top-$k$ support tree ranking} (\emph{for $N$}) is defined to be a unique subsequence of the first $k$ elements of $(\Omega, \leq^*)$. 
Note that for any $k\in \mathbb{N}$, one can  determine  in $O(|A(N)|)$ time  whether or not $k\leq |\Omega|$ holds~\cite{MH2018structural}.

\begin{problem}
\textbf{Top-$k$ support tree ranking problem}\label{prob}\\ 
\textbf{Input: } 
A tree-based phylogenetic $X$-network $N$ with associated probability $w: A(N)\rightarrow (0,1]$ and $k\in \mathbb{N}$ not exceeding the number $|\Omega|$ of support trees of $N$.\\
\textbf{Output: } 
The top-$k$ support tree ranking $\langle \tau^{(1)},\dots, \tau^{(k)}\rangle$ for $N$.
\end{problem}

\section{Results}\label{sec:results}
As a preliminary step, we prove the following proposition about the local ranking. 
\begin{proposition}\label{prop:preprocessing}
For any  maximal zig-zag trail $Z_i$ in a tree-based phylogenetic $X$-network $N$ with associated probability $w: A(N)\rightarrow (0,1]$, the first element in  the local ranking $(\Omega_i, \leq^*)$ can be found in $O(|A(Z_i)|)$ time. Moreover, given the $j$-th  element in $(\Omega_i, \leq^*)$, one can find the $(j+1)$-th  element in $O(|A(Z_i)|)$ time. 
\end{proposition}
\begin{proof}
One can check in $O(|A(Z_i)|)$ time whether $Z_i$ is a crown, N-fence or M-fence. In the case when $Z_i$ is a crown or N-fence, the local ranking for $Z_i$ is trivial to compute  as $|\Omega_i|\leq 2$ holds by  Theorem~\ref{structure}. Assume that $Z_i$ is an M-fence $\langle a_1,\dots,a_{2m}\rangle$ with $|A(Z_i)|=2m$. Also, let $x_{p+1}:=(1 (01)^p (10)^q 1)$ with  $p+q=m-1$ for each $p\in [0, m-1]$ and let $\Delta_p:=w(a_{2p+1})-w(a_{2p})$ for each $p\in [1, m-1]$. 
Then, $f(x_{p+1})=f(x_p)+\Delta_p$ holds for each $p\in [1, m-1]$. As one can obtain both $f(x_1)$ and $\langle \Delta_1,\dots, \Delta_{m-1}\rangle$ in $O(|A(Z_i)|)$ time, computing the likelihood values $f(x)$ for all $x\in \Omega_i$ requires  $O(|A(Z_i)|)$ time. This completes the proof. 
\end{proof}

We define $I_0 := \emptyset$ and $I_{j} := \{\tau^{(0)}, \dots, \tau^{(j)}\}$ for each $j\in[1, k]$. 
 Recalling $\Omega=\prod_{i=1}^{d}{\{1,\dots, |\Omega_i|\}}$, we see that $(\Omega, \leq^*)$ is a linear extension of the partially ordered set $(\Omega, \leq)$ (\textit{i.e.}, $x \leq y$ implies $x \leq^* y$), where $\leq$ is the usual component-wise order on vectors (\textit{e.g.}, $(x_1\; x_2)\leq (y_1\; y_2)$ if and only if $x_1\leq y_1$ and $x_2\leq y_2$). 
 We also note that this requires each $I_j$ to be an order ideal of $(\Omega, \leq)$ (\textit{i.e.}, for any $x\leq y$, $y\in I_j$  implies $x\in I_j$). 
 These arguments lead to the following proposition.
 
\begin{proposition}
Let $\langle \tau^{(1)},\dots, \tau^{(k)}\rangle$  be the top-$k$ support tree ranking for a tree-based phylogenetic $X$-network $N$  with associated probability $w: A(N)\rightarrow (0,1]$ and let $I_j$ be as defined above.
Then, $I_1=\{(1\; \dots\; 1)\}$ holds, and  for each $j\in [1,k-1]$, there exists $\tau\in I_j$ with $\| \tau^{(j+1)} - \tau \|_1 = 1$.	
\end{proposition}

Let $e_i$ be the unit vector such that $i$-th component is one and the others are all zeros. Also, for each $\tau\in \Omega\setminus \{\tau^{(1)}\}$, let $\mathrm{id}(\tau)$ be the first index such that the $i$-th component of $v$ is  strictly greater than one and let $e(\tau):=e_{\mathrm{id}(\tau)}$. 
For example, $\tau=(1\; 1\; 1\; 5\; 8)$ gives $e(\tau)=(0\; 0\; 0\; 1\; 0)$. Then, we have the next lemma, which is illustrated in Figure~\ref{fig:threefigs}. 

\begin{lemma}\label{lem:graph.representation}
Let $(\Omega, \leq^*)$ be the support tree ranking for a tree-based phylogenetic $X$-network $N$  with associated probability $w: A(N)\rightarrow (0,1]$ and let  $\Gamma$ be a graph with $V(\Gamma)=\Omega$ and 
$A(\Gamma)=\{(\tau, \tau^\prime)\in \Omega\times (\Omega\setminus \{\tau^{(1)}\}) \mid \tau = \tau^\prime - e(\tau^\prime)\}$. 
Then, $\Gamma$  is a spanning tree of the Hasse diagram of $(\Omega, \leq)$ such that  $\tau^{(1)}$  is the root of $\Gamma$ and $(\tau, \tau^\prime)\in A(\Gamma)$ implies $\tau\leq^* \tau^\prime$. 
\end{lemma} 
\begin{proof}
It is clear that $(\tau, \tau^\prime)\in A(\Gamma)$ implies $\tau\leq \tau^\prime$ (and hence $\tau\leq^* \tau^\prime$). By construction, $\Gamma$ is a tree rooted at $\tau^{(1)}$ because ${\it deg}^-_\Gamma(\tau^{(1)})=0$ holds and for each  $\tau^\prime \in V(\Gamma)\setminus \{\tau^{(1)}\}$, there exists a unique element $\tau\in V(\Gamma)$ with $(\tau, \tau^\prime)\in A(\Gamma)$. This completes the proof.
\end{proof}

\begin{figure}[htbp]
\centering
\includegraphics[scale=.7]{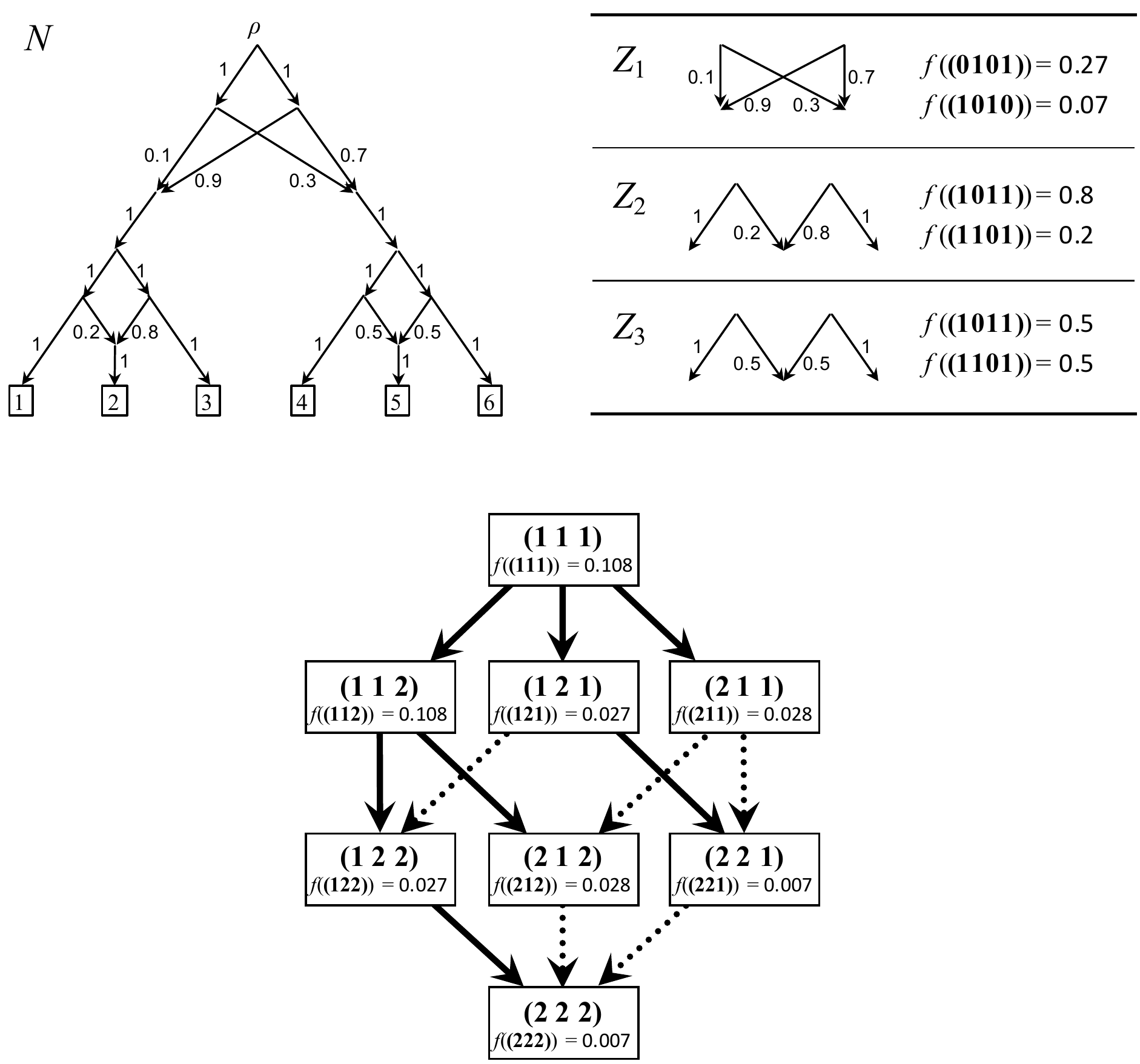}
\caption{An illustration of Lemma~\ref{lem:graph.representation}. The top left is a tree-based phylogenetic $X$-network $N$ whose arcs are associated with probability. 
The top right shows the maximal zig-zag trails $Z_i$ in $N$ with $|\Omega_i|\geq 2$
 and the likelihood of each element of $(\Omega_i, \leq^*)$ ($i=1, 2, 3$). On the bottom is the spanning tree $\Gamma$ (shown in bold) of the Hasse diagram of   $(\Omega, \leq)$.
\label{fig:threefigs}}
\end{figure}

In what follows, for any $\Omega^\prime \subseteq \Omega$, we write $\it least(\Omega^\prime)$ to mean the least element of $(\Omega^\prime, \leq^*)$. For any $\tau\in \Omega$,  let ${\it Child}(\tau):=\{\tau^\prime\in V(\Gamma) \mid (\tau, \tau^\prime)\in A(\Gamma)\}$ and $\it child^*(\tau):=least(Child(\tau))$. Also, for any $\tau^\prime\in \Omega \setminus \{\tau^{(1)}\}$,  let  $\it sibling^*(\tau^\prime):=least(\{\tau\in Child(parent(\tau^\prime)) \mid \tau^\prime \leq^* \tau \land \tau \neq \tau^\prime\})$, where ${\it parent}(\tau^\prime)$ represents a unique element $\tau \in V(\Gamma)$  with $(\tau, \tau^\prime)\in A(\Gamma)$.  We note that both $\it child^*(\tau^\prime)=\emptyset$ and $\it sibling^*(\tau^\prime)=\emptyset$ are possible to occur.

\begin{lemma}\label{recursionQ}
Let $\Gamma$ be the graph as in Lemma~\ref{lem:graph.representation} and let $Q_{j}$ be a subset of $V(\Gamma)$ that is recursively defined by  
\begin{equation}\label{eq:Q}
 Q_{j}:=
 \begin{cases}
(Q_{j-1}\setminus \{\tau^{(j-1)}\})\cup \{\it child^*(\tau^{(j-1)}), \it sibling^*(\tau^{(j-1)})\} & j \in[2,k] \\
 \{\tau^{(1)}\} & j=1. 	
 \end{cases}
 \end{equation}
Then, for  each $j\in [1, k]$, we have  $\tau^{(j)} \in Q_{j}$ and  $\tau^{(\ell)} \not\in Q_{j}$ for all $\ell < j$.   
\end{lemma}
\begin{proof}
Let $P_1=Q_1$ and $P_{j}=\{\it least(Child(\tau)\setminus I_{j-1})\mid \tau\in I_{j-1}\}$ for  $j\in [2, k]$. 
We will show that $P_j = Q_j$ holds for any $j\in [1, k]$, which completes the proof, since 
 $\tau^{(j)} \in P_{j}$ and  $\tau^{(\ell)} \not\in P_{j}$ for all $\ell < j$. 

For $j \in [2, k]$, we have 
 \begin{align*}
 P_{j}
&=\{\it least(Child(\tau)\setminus I_{j-1})\mid \tau\in I_{i-2}\} \cup \{\it least(Child(\tau^{(j-1)})\setminus I_{j-1})\}\\ 
  &=\{\it least(Child(\tau)\setminus I_{j-1})\mid \tau\in I_{j-2}\} \cup \{\it child^*(\tau^{(j-1)})\}, 
  \end{align*}
where we assume that $I_0=\emptyset$. 
 Note that $I_{j-2}$ contains $p:={\it parent(\tau^{(j-1)})}$. This implies  
  \begin{align*} 
 P_{j}
 &=\{\it least(Child(\tau)\setminus I_{j-1})\mid \tau\in I_{j-2}\setminus \{p\}\} \cup \{\it child^*(\tau^{(j-1)}), least(Child(p)\setminus I_{j-1})\}\\
 &=\{\it least(Child(\tau)\setminus I_{j-1})\mid \tau\in I_{j-2}\setminus \{p\}\} \cup \{\it child^*(\tau^{(j-1)}), sibling^*(\tau^{(j-1)})\}.
 \end{align*}
 For any $\tau\in I_{j-2}\setminus \{p\}$, we have $\it least(Child(\tau)\setminus I_{j-1})=\it least(Child(\tau)\setminus I_{j-2})$ because $\tau^{(j-1)}\not \in \it Child(\tau)$ holds. We thus obtain
 \begin{align*}
P_{j}
 &=(\{\it least(Child(\tau)\setminus I_{j-2}) \mid \tau\in I_{j-2}\}\setminus \{\it least(Child(p)\setminus I_{j-2})\}) \cup \{\it child^*(\tau^{(j-1)}), \it sibling^*(\tau^{(j-1)})\}\\
 &=(P_{j-1}\setminus \{\tau^{(j-1)}\}) \cup \{\it child^*(\tau^{(j-1)}), \it sibling^*(\tau^{(j-1)})\}. 
 \end{align*}
From Equation~(\ref{eq:Q}) and $P_1=Q_1$, the desired conclusion follows. 
\end{proof}

We are in a position to give an algorithm for Problem~\ref{prob}. As illustrated in Table~\ref{table}, the algorithm starts by setting $j:=1$ and $Q_1:=\{\tau^{(1)}\}$ and  then    returns $\tau^{(j)}={\it least}(Q_{j})$ for each $j\in [1, k]$, where $Q_j$ is iteratively updated using  Equation~(\ref{eq:Q}).   

\begin{table}[htbp]
\caption{Application of the proposed algorithm to  the input $N$   in Figure~\ref{fig:threefigs} ($k=8$).}
	\label{table}
	\centering
	\scalebox{0.8}{
\begin{tabular}{lcccccccc}
\hline
                            & $j=1$         & $j=2$ & $j=3$         & $j=4$         & $j=5$         & $j=6$ & $j=7$         & $j=8$         \\ \hline
$Q_j$                       & $\{(1\; 1\;1)\}$         & $\{(1\; 1\; 2)\}$ & $\{(2\;1\;1), (2\;1\;2)\}$    & $\{(1\;2\;1), (2\;1\;2)\}$    & $\{(1\;2\;1), (1\;2\;2)\}$         & $\{(1\;2\;2), (2\;2\;1)\}$ & $\{(2\;2\;1), (2\;2\;2)\}$    & $\{(2\;2\;2)\}$         \\
$\tau^{(j)}$                & $(1\;1\;1)$         & $(1\;1\;2)$ & $(2\;1\;1)$         & $(2\;1\;2)$         & $(1\;2\;1)$         & $(1\;2\;2)$ & $(2\;2\;1)$         & $(2\;2\;2)$         \\
$\it child^*(\tau^{(j)})$   & $(1\;1\;2)$         & $(2\;1\;2)$ & $\emptyset$ & $\emptyset$ & $(2\;2\;1)$         & $(2\;2\;2)$ & $\emptyset$ &  \\
$\it sibling^*(\tau^{(j)})$ & $\emptyset$ & $(2\;1\;1)$ & $(1\;2\;1)$         & $(1\;2\;2)$ & $\emptyset$ & $\emptyset$ & $\emptyset$ &  \\ \hline
\end{tabular}
}
\end{table}

In order to analyse the running time of the above algorithm, 
let us review some basics of a \emph{priority queue}, which is a data structure for maintaining objects that are prioritised by their associated values. In its most basic form, a priority queue supports the operations called \textsc{Insert} and \textsc{Delete-min}, where the former refers to adding a new object, and the latter to detecting and deleting the  one with the highest-priority~\cite{introduction2algms}. 
 Implemented with a binary heap, each of these operations can be performed in $O(\log n)$ time, where $n$ denotes the number of the elements in the priority queue~\cite{introduction2algms}. 

\begin{theorem}\label{thm:main}
The top-$k$ support tree ranking problem (Problem~\ref{prob}) can be solved with linear delay, and hence in $O(k|A(N)|)$ time.
\end{theorem}
\begin{proof}
As  Equation~(\ref{eq:Q}) implies that  $|Q_{j+1}-Q_j|\leq 1$ holds for any $j\in[1, k-1]$,  $|Q_j| \leq k$ holds for any $j\in [1,k]$. 
Then, if we keep the elements of each $Q_{j}$ in a priority queue,  $O(\log k)$ time suffices  to  return $\tau^{(j)}$ and to delete $\tau^{(j)}$ from $Q_{j}$. 
Also, once $\it child^*(\tau^{(j)})$ and $\it sibling^*(\tau^{(j)})$ have been obtained, inserting the two elements requires $O(\log k)$ time. We note that $O(\log k)\leq O(|A(N)|)$ follows from $k\leq 2^{|A(N)|}$. 
By Proposition~\ref{prop:preprocessing}, for each $j\in[1, k-1]$, one can compute $\{\it child^*(\tau^{(j)}), \it sibling^*(\tau^{(j)})\}$ in $\sum_{i=1}^{d}{O(|A(Z_i)|)}$ time, which equals $O(|A(N)|)$ time as $\{Z_1,\dots,Z_d\}$ is a decomposition of $N$. Hence, our algorithm can return $\tau^{(1)},\dots, \tau^{(k)}$ one after the other in such a way that the delay between two consecutive outputs is $O(|A(N)|)$ time. This completes the proof.
\end{proof}

Finally, we make two remarks. First, $\Omega(k|A(N)|)$ time is required to output $k$ distinct support trees of $N$ as each support tree has size $\Omega(|A(N)|)$. Therefore, the running time of our algorithm (as well as that of the enumeration algorithm in~\cite{MH2018structural}) is $\Theta(k|A(N)|)$, which guarantees the optimality of those algorithms. Second, as commonly in the literature (\textit{e.g.},~\cite{kapoor1995algorithms}), it would be natural to wonder about the time complexity of an analogue of Problem~\ref{prob} that only asks for outputting a sequence of the differences between $\tau^{(j-1)}$ and $\tau^{(j)}$; however, we note that this problem still requires $\Omega(k|A(N)|)$ time because the size of each difference is $\Omega(|A(N)|)$. To illustrate this, consider a tree-based phylogenetic $X$-network  $N$ that is decomposed into maximal fences, each of which has only one admissible arc-set, and $c$ crowns, each of which has size $\Omega(|A(N)|/c)$. The difference between any two support trees has size  $\Omega(|A(N)|/c)$, which equals $\Omega(|A(N)|)$ if $c$ is a constant. 

%
%

\bibliographystyle{amsplain}

\bibliography{kbestsubdivtrees.bib}

\section*{Acknowledgement}
The first author acknowledges support from JST PRESTO Grant Number JPMJPR16EB.

\end{document}